\documentclass[11pt,twoside]{amsart}
\usepackage{amsmath}
\usepackage{amssymb}
\usepackage{amsthm}
\usepackage{latexsym}
\usepackage{amscd}
\usepackage{xypic}
\xyoption{curve}
\usepackage{ifthen} 
\usepackage{hyperref} 
\usepackage{graphicx}
\usepackage{enumerate}
\usepackage{enumitem}

\usepackage{enumitem}

 \def\cocoa{{\hbox{\rm C\kern-.13em o\kern-.07em C\kern-.13em o\kern-.15em A}}}

\usepackage{xcolor}

\newtheorem{theorem}{Theorem}[section]
\newtheorem{question}[theorem]{Question}
\newtheorem{lemma}[theorem]{Lemma}
\newtheorem{proposition}[theorem]{Proposition}

\theoremstyle{definition}
\newtheorem{remark}[theorem]{Remark}
\newtheorem{definition}[theorem]{Definition}

\newcommand {\gr}{\mathrm{gr}}

\newcommand {\Hom}{\mathcal{H}\kern -0.25ex{\mathit om}}
\newcommand {\Spl}{\mathcal{S}\kern -0.25ex{\mathit pl}}
\newcommand {\Ext}{\mathcal{E}\kern -0.25ex{\mathit xt}}

\newcommand {\rk}{\mathrm{rk}}

\newcommand {\ext}{\mathrm{Ext}}
\newcommand {\Hilb}{\mathcal{H}\kern -0.25ex{\mathit ilb\/}}

\newcommand {\Ss}{\mathcal{S}\kern -0.25ex{\mathit s}}

\newcommand {\bZ}{\mathbb{Z}}

\newcommand {\bC}{\mathbb{C}}
\newcommand {\bP}{\mathbb{P}}

\newcommand{\cE}{{\mathcal E}}
\newcommand{\cF}{{\mathcal F}}
\newcommand{\cM}{{\mathcal M}}

\newcommand{\cO}{{\mathcal O}}

\newcommand{\Pic}{\operatorname{Pic}}

\newcommand{\Num}{\operatorname{Num}}

\def\p#1{{\bP^{#1}}}

\def\ga#1{{{\accent"12 #1}}}
 
\title[Stability of rank $2$ Ulrich bundles on projective $K3$ surface]{Stability of rank $2$ Ulrich bundles\\  on projective $K3$ surface}



\author[Gianfranco Casnati, Federica Galluzzi]{Gianfranco Casnati, Federica Galluzzi}
\thanks{The authors are members of GNSAGA group of INdAM and are supported by the framework of PRIN 2010/11 \lq Geometria delle variet{\accent"12 a} algebriche\rq, cofinanced by MIUR}

\begin{document}

\begin{abstract}
Let $F\subseteq\p{N}$ be a $K3$ surface of degree $2a$, where $a\ge2$. In this paper we deal with Ulrich bundles on $F$ of rank $2$. We deal with their stability and we construct $K3$ surfaces  endowed with families of non--special Ulrich bundles of rank $2$ for each $a\ge2$.
\end{abstract}

\maketitle

\section{Introduction and Notation}
Throughout the whole paper, $\p N$ will denote the projective space of dimension $N$ over the complex field $\bC$. 

Each smooth surface $F\subseteq\p N$ is endowed with a polarization $\cO_F(h):=\cO_{\p N}(1)\otimes\cO_F$. A natural problem in the study of the geometry of $F$ is to deal with the vector bundles that it supports. 

Clearly we can restrict our attention to {\sl indecomposable} bundles, i.e. bundles which do not split as a direct sum of bundles of lower rank. From the cohomological viewpoint, the simplest vector bundles are the {\sl arithmetically Cohen--Macaulay} ({\sl aCM} for short) ones, i.e. bundles $\cE$ such that $h^1\big(F,\cE(th)\big)=0$ for $t\in\bZ$. Notice that such a property is trivially invariant up to shifting degrees. Thus we can focus on {\sl initialized} bundles, i.e. bundles $\cE$ such that $h^0\big(F,\cE(-h)\big)=0$ and $h^0\big(F,\cE\big)\ne0$. 

Horrocks' theorem (see \cite{O--S--S} and the references therein) guarantees that $\cO_F$ is the unique initialized, indecomposable, aCM bundle when $F\subseteq\p N$ is a plane. Recall that a closed subscheme $F\subseteq\p N$ is called aCM if it is projectively normal and $\cO_F$ is aCM. A very general result of D. Eisenbud and J. Herzog (see \cite{E--He}) implies that, besides planes, only few other surfaces support at most a finite number of aCM bundles, namely smooth quadrics, smooth rational scrolls of degree up to $4$, the Veronese surface.  

M. Casanellas and R. Hartshorne proved in \cite{C--H1} and \cite{C--H2} that a smooth cubic surface in $\p3$ is endowed with families of arbitrary dimension of non--isomorphic, indecomposable initialized aCM bundles. In order to achieve their results, the authors constructed families of initialized aCM bundles with an extra property. Indeed they are {\sl Ulrich bundles}, i.e. bundles $\cE$ on $F$ whose minimal free resolution as sheaves on $\p3$ is linear. To complete the picture in the case of a cubic surface, we recall that D. Faenzi completely described aCM bundles of ranks $1$ and $2$ in \cite{Fa}.

Some results are known also for quartic surfaces $F\subseteq\p3$. K. Watanabe classified in \cite{Wa} aCM line bundles on $F$, identifying Ulrich line bundles. E. Coskun, R. Kulkarni, Y. Mustopa proved in \cite{C--K--M1} that such an $F$ always supports a family of dimension $14$ of Ulrich bundles of rank $2$ with first Chern class $\cO_F(3h)$. As a consequence of the existence of these bundles one can also infer that $F$ is linear pfaffian, i.e. the quartic polynomial defining $F$ is the pfaffian of a $8\times8$ skew--symmetric matrix with linear entries.

Another almost immediate application of this existence result is that such a surface also  supports families of arbitrary dimension of non--isomorphic, indecomposable Ulrich bundles (see the note \cite{Cs3}). In \cite{C--N} other interesting families of initialized aCM bundles of rank $2$ are constructed: their existence implies that $F$ is quadratic pfaffian, i.e. the quartic polynomial defining $F$ is also the pfaffian of a $4\times4$ skew--symmetric matrix with quadratic entries. Finally, the complete description of initialized aCM bundles of rank $2$ on each linear determinantal smooth quartic surface, i.e. a surface defined by a quartic polynomial which is the determinant of a $4\times4$ matrix with linear entries, is exploited in \cite{Cs2}.
As far as we know there are no other general results for smooth quartic surfaces.

When the codimension increases the picture becomes quickly vague. E.g. even for del Pezzo surfaces only scattered results are known: for this class of surfaces J. Pons Lopis and F. Tonini studied aCM line bundles in \cite{PL--T}, while E. Coskun, R.S. Kulkarni, Y. Mustopa gave in \cite{C--K--M2}, among other results, restrictions on the first Chern class of Ulrich bundles.

Notice that quartic surfaces are a particular case of {\sl $K3$ surfaces}, i.e. smooth regular surfaces $F$ such that $\omega_F\cong\cO_F$. These surfaces can be embedded in $\p{a+1}$ as non--degenerate aCM surfaces of degree $2a$ for some $a\ge2$ (see \cite{SD} for the details).  M. Aprodu, G. Farkas, 
A. Ortega generalized in \cite{A--F--O} the results of \cite{C--K--M1} to this family of $K3$ surfaces $F$, under an extra technical condition. They construct therein families of rank $2$ Ulrich bundles with first Chern class $\cO_F(3h)$: following Proposition 6.2 of \cite{E--S--W}, they call such bundles {\sl special}.  

As we already pointed out, the role of Ulrich bundles is particularly important, hence we ask for further informations on them. For example, it would be interesting to answer the following questions.
\begin{enumerate}[label=(\Alph*)]
\item Are there  restrictions on the Chern classes of Ulrich bundles on a $K3$ surface?
\item Are there other rank $2$ indecomposable Ulrich bundles on a $K3$ surface besides the special ones described in \cite{C--K--M1} and \cite{A--F--O}?
\item Which are the (semi)stability properties of a rank $2$ indecomposable Ulrich bundle?
\end{enumerate}

In this paper we give partial answers to the questions listed above.  E.g., at the end of Section \ref{sUlrich} we prove the following easy proposition answering the first question (see Proposition \ref{pBound}).

\medbreak
\noindent
{\bf Partial answer to Question A.}
{\it Let $\cE$ be an Ulrich bundle of rank $r$ on a $K3$ surface $F\subseteq\p{N}$ of degree $2a$, where $a\ge2$.

Then $c_1(\cE)^2$ is an even integer satisfying
$$
4(a-1)r^2\le c_1(\cE)^2\le\frac{9}{2}ar^2,
$$
where $c_1(\cE)^2\ne\frac{9}{2}ar^2-2$ if $r$ is even. Moreover, $c_1(\cE)^2=\frac{9}{2}ar^2$ if and only if $c_1(\cE)=3rh/2$. 

If $\cE$ is simple, then $(4a-2)r^2-2\le c_1(\cE)^2$.}

\medbreak

We also show that both the above upper and lower bounds are trivially sharp. Notice that simple bundles of rank $r$ are trivially indecomposable: the converse is also true for Ulrich bundles when $r=2$ (see Lemma \ref{lRank2}). Using such an equivalence we are also able to answer the second question in Section \ref{sUlrich2} proving that all the intermediate values which are admissible for simple bundles are actually attained (see Theorem \ref{tExistence}) on suitable $K3$ surfaces. The bundles that we construct are non--special in the sense of \cite{A--F--O}, i.e. their first Chern class is not $\cO_F(3h)$.

\medbreak
\noindent
{\bf Partial answer to Question B.}
{\it Let $a\ge2$. For each choice of an integer $u$ in the range $4a-1\le u\le 5a+4$, $u\ne 5a+3$, there exists a $K3$ surface $F\subseteq\p N$ of degree $2a$ and $\rk(\Pic(F))=3$ supporting an indecomposable Ulrich bundle $\cE$ of rank $2$ with $c_1(\cE)^2=8a-8+2u$ and $c_2(\cE)=u$.}

\medbreak

Notice that when $\rk(\Pic(F))=1$ only the bundles with $u=5a+4$ can actually exist on $F$ and they are exactly the aforementioned special bundles constructed in \cite{C--K--M1} and \cite{A--F--O} (see \cite{Co}: see also \cite{Cs1}). 

Our construction cannot be extended to obtain bundles on a $K3$ surface $F\subseteq\p N$ of degree $2a$ with $\rk(\Pic(F))=2$: in particular we are unfortunately unable to prove or disprove the existence of these bundles on surfaces satisfying $\rk(\Pic(F))=2$. Moreover, as pointed out in Section 2 of \cite{E--S--W}, the Chow forms of the surfaces we use are always linear determinantal. Thus the problem of the existence of non--special rank $2$ Ulrich bundles on $K3$ surfaces whose Chow form is not determinantal remains wide open.

Finally, in Section \ref{sStable}, we answer the third question raised above. Indeed we prove therein the following result (see Theorem \ref{tStable}).

\medbreak
\noindent
{\bf Partial answer to Question C.}
{\it Let $F\subseteq\p{N}$ be a $K3$ surface of degree $2a$, where $a\ge2$.

If $\cE$ is an indecomposable  Ulrich bundle of rank $2$ on $F$ which is  strictly semistable and general in its moduli space, then it fits into an exact sequence of the form
\begin{equation}
\label{seqAB}
0\longrightarrow\cO_F(A)\longrightarrow\cE\longrightarrow\cO_F(B)\longrightarrow0.
\end{equation}
where $\cO_F(A)$ and $\cO_F(B)$ are Ulrich line bundles on $F$ such that $AB=4a-1$.  In particular $c_1(\cE)^2=16a-10$ and $c_2(\cE)=4a-1$.}

\medbreak

It follows that the general Ulrich bundle $\cE$ of rank $2$ with fixed Chern classes $c_1(\cE)$ and $c_2(\cE)$ on a $K3$ surface $F\subseteq\p{N}$ of degree $2a$ is always stable when $c_2(\cE)\ne4a-1$.  Also in this case our answer is partial: indeed we are not able to deal with the stability properties of each general Ulrich bundle $\cE$ such that $c_2(\cE)=4a-1$. 

In Section \ref{sGeneral} we recall the results that we need in the paper on Ulrich bundles. In Section \ref{sQuartic} we summarize several facts about $K3$ surfaces. Section \ref{sUlrich} is devoted to list and inspect some properties of Ulrich bundles on $K3$ surfaces. In Section \ref{sUlrich2} we focus on rank $2$ bundles. Finally, in Section \ref{sStable} we deal with the stability properties of rank 2 Ulrich bundles.

The authors would like to thank A. Knutsen for some helpful suggestions. Particular thanks go to the referee, whose many comments  and interesting suggestions considerably improved the paper.

\section{General results on Ulrich bundles}
\label{sGeneral}

In this section we summarize some general results on Ulrich bundles on a smooth, irreducible, closed subscheme $X\subseteq \p N$. In what follows we will always set $\cO_X(h):=\cO_{\p N}(1)\otimes\cO_X$.

\begin{definition}
Let $X\subseteq \p N$ be a smooth irreducible closed subscheme and let $\cF$ be a vector bundle on $X$.

We say that:
\begin{itemize}
\item $\cF$ is initialized if $h^0\big(X,\cF(-h)\big)=0\ne h^0\big(X,\cF\big)$.
\item $\cF$ is aCM if $h^i\big(X,\cF(th)\big)=0$ for each $t\in \bZ$ and $i=1,\dots,\dim(X)-1$.
\item $\cF$ is Ulrich if $h^i\big(X,\cF(-ih)\big)=h^j\big(X,\cF(-(j+1)h)\big)=0$ for each $i>0$ and $j<\dim(X)$.
\end{itemize}
\end{definition}

Ulrich bundles collect many interesting properties (see Section 2 of \cite{E--S--W}). E.g. they are automatically initialized, aCM and globally generated. Every direct summand of an Ulrich bundle is Ulrich as well. Finally, as already pointed out in the introduction, $\cF$ is Ulrich if and only if it has a linear minimal free resolution over $\p N$. 
Ulrich bundles also behave well with respect to the notions of (semi)stability and $\mu$--(semi)stability. Recall that for each bundle $\cF$ on $X$, the slope $\mu(\cF)$ and the reduced Hilbert polynomial $p_{\cF}(t)$ (with respect to $\cO_X(h)$) are defined as follows:
$$
\mu(\cF)= c_1(\cF)h^{\dim (X)-1}/\rk(\cF), \qquad p_{\cF}(t)=\chi(\cF(th))/\rk(\cF).
$$
The bundle $\cF$ is $\mu$--semistable (resp. $\mu$--stable) if for all subsheaves
$\mathcal G$ with $0<\rk(\mathcal G)<\rk(\cE)$ we have $\mu(\mathcal G) \le \mu(\cE)$ (resp. $\mu(\mathcal G)< \mu(\cE)$).

The bundle $\cE$ is called semistable (resp. stable) if for all $\mathcal G$ as above $p_{\mathcal G}(t) \le  p_{\cE}(t)$ (resp. $p_{\mathcal G}(t) <  p_{\cE}(t)$) for $t\gg0$. We recall that in order to check the semistability and stability of a bundle one can restrict the attention only to subsheaves such that the quotient is torsion--free.

The following chain of implications holds for $\cF$:
$$
\text{$\cF$ is $\mu$--stable}\ \Rightarrow\ \text{$\cF$ is stable}\ \Rightarrow\ \text{$\cF$ is semistable}\ \Rightarrow\ \text{$\cF$ is $\mu$--semistable.}
$$

For the following result see Theorem 2.9 of \cite{C--H2}.

\begin{theorem}
\label{tUnstable}
Let $X\subseteq\p N$ be a smooth, irreducible closed subscheme.

If $\cE$ is an Ulrich bundle on $X$ the following assertions hold.
\begin{enumerate}[label=(\alph*)]
\item $\cE$ is semistable and $\mu$--semistable.
\item $\cE$ is stable if and only if it is $\mu$--stable.
\item If
\begin{equation}
\label{seqUnstable}
0\longrightarrow\mathcal L\longrightarrow\cE\longrightarrow\mathcal M\longrightarrow0
\end{equation}
is an exact sequence of coherent sheaves with $\cM$ torsion free and $\mu(\mathcal L)=\mu(\cE)$, then both $\mathcal L$ and $\cM$ are Ulrich bundles.
\end{enumerate}
\end{theorem}

We conclude this section with the following helpful result.

\begin{lemma}
\label{lRank2}
Let $X\subseteq\p N$ be a smooth, irreducible closed subscheme with $h^1\big(X,\cO_X\big)=0$.

If $\cE$ is an Ulrich bundle of rank $2$ on $X$, then $\cE$ is simple if and only if it is indecomposable.
\end{lemma}
\begin{proof}
If $\cE$ is simple, then it is trivially indecomposable. Conversely, assume that $\cE$ is indecomposable. If it is $\mu$--stable, then it is simple (see \cite{H--L}, Corollary 1.2.8). 

Assume that $\cE$ is strictly semistable. In particular $\cE$ fits into Sequence \eqref{seqUnstable} with $\cM$ torsion--free and $\mu(\mathcal L)=\mu(\cE)$. It follows from Theorem \ref{tUnstable} that $\mathcal L\cong\cO_X(A)$ and $\cM\cong\cO_X(B)$ are Ulrich line bundles, $\mu(\cO_X(A))=\mu(\cO_X(B))$. Thus $\cE$ fits into the exact sequence 
\begin{equation*}
\label{seqExtAB}
0\longrightarrow\cO_X(A)\longrightarrow\cE\longrightarrow\cO_X(B)\longrightarrow0.
\end{equation*}
If $\cO_X(A)\cong \cO_X(B)$ the above sequence would split because $h^1\big(X,\cO_X(A-B)\big)=h^1\big(X,\cO_X\big)$ in this case. Thus we can assume $\cO_X(A)\not\cong \cO_X(B)$.

The above sequence splits if and only if it corresponds to $0\in \ext_F^1\big(\cO_X(B),\cO_X(A)\big)$, thus if and only if $\cE$ is not simple, due to Proposition 5.3 of \cite{PL--T}. Since we are assuming that $\cE$ is indecomposable, it follows that $\cE$ is simple.
\end{proof}

\section{General results on $K3$ surfaces}
\label{sQuartic}
We recall some facts on a $K3$ surface $F\subseteq \p N$ with hyperplane line bundle $\cO_F(h)$. They are collected from several places (e.g. see \cite{SD} and \cite{B--H--P--VV}). 

We know that $\omega_F\cong\cO_F$ and $q(F)=0$. In particular $p_a(F)=p_g(F)=1$.
The first important fact is that Serre duality for each locally free sheaf $\cF$ on $F$ becomes
$$
h^i\big(F,\cF\big)=h^{2-i}\big(F,\cF^\vee\big),\qquad i=0,1,2.
$$
Moreover the Riemann--Roch theorem on $F$ is
\begin{equation}
\label{RRGeneral}
h^0\big(F,\cF\big)+h^{2}\big(F,\cF\big)=h^{1}\big(F,\cF\big)+2\rk(\cF)+\frac{c_1(\cF)^2}2-c_2(\cF).
\end{equation}
In particular, if $\cF\cong\cO_F(D)$ for a divisor $D$ with $D^2\ge-2$, then either $D$ or $-D$ is necessarily effective.

If $D$ is an effective non--zero divisor on $F$, then $h^2\big(F,\cO_F(D)\big)=h^0\big(F,\cO_F(-D)\big)=0$. Moreover
$$
h^1\big(F,\cO_F(D)\big)=h^1\big(F,\cO_F(-D)\big)=h^0\big(D,\cO_D\big)-1
$$
(see \cite{SD}, Lemma 2.2). It follows that 
\begin{equation}
\label{h^1}
h^0\big(F,\cO_F(D)\big)=2+\frac {D^2}2,\qquad \deg(D)=Dh,\qquad p_a(D)=1+\frac {D^2}2,
\end{equation}
for each integral divisor $D$ on $F$ (see \cite{SD}, Paragraph 2.4). In particular, the  integral fixed divisors $D$ satisfy $D^2=-2$ and $D\cong\p1$.

We summarize the other helpful results we will need in the following statement.

\begin{proposition}
\label{SD}
Let $F$ be a $K3$ surface.

For each effective divisor $D$ on $F$ such that $\vert D\vert$ has no fixed components the following assertions hold. 
\begin{enumerate}[label=(\alph*)]
\item $D^2\ge0$ and $\cO_F(D)$ is globally generated.
\item If $D^2>0$, then the general element of $\vert D\vert$ is irreducible and smooth: in this case $h^1\big(F,\cO_F(D)\big)=0$.
\item If $D^2=0$, then there is an irreducible divisor $\overline{D}$ with $p_a(\overline{D})=1$ such that $\cO_F(D)\cong\cO_F(e\overline{D})$ where $e-1:=h^1\big(F,\cO_F(D)\big)$: in this case the general element of $\vert D\vert$ is smooth.
\end{enumerate}
\end{proposition}
\begin{proof}
See \cite{SD}, Proposition 2.6 and Corollary 3.2.
\end{proof}

Let $F\subseteq \p N$ be a $K3$ surface with hyperplane line bundle $\cO_F(h)$ such that $h^2=2a$, where $a\ge2$. It would be interesting to classify all the aCM line bundles on $F$. If $h^2=4$ a complete classification can be found in  \cite{Wa}. A similar classification for double covers of $\p2$ can be found in \cite{Wa2}.

The problem of identifying aCM line bundles on $K3$ surfaces is by no way trivial, as one can check by looking at the quoted papers.

\section{Ulrich bundles on $K3$ surfaces}
\label{sUlrich}
In this section we will prove some general preliminary results about Ulrich bundles on $K3$ surfaces $F\subseteq\p N$ of degree $h^2=2a$, where $a\ge2$, giving a partial answer to the first question raised in the introduction.  

\begin{lemma}
\label{lUlrich}
Let $F\subseteq\p{N}$ be a $K3$ surface of degree $2a$, where $a\ge2$.

The following assertions are equivalent for a vector bundle $\cE$ of rank $r$ on $F$:
\begin{enumerate}[label=(\alph*)]
\item $\cE$ is Ulrich;
\item $\cE^\vee(3h)$ is Ulrich;
\item $\cE$ is aCM and 
\begin{equation}
\label{eqUlrich}
c_1(\cE)h=3ar,\qquad c_2(\cE)=\frac{c_1(\cE)^2}2-2(a-1)r;
\end{equation}
\item $h^0\big(F,\cE(-h)\big)=h^0\big(F,\cE^\vee(2h)\big)=0$ and Equalities \eqref{eqUlrich} hold.
\end{enumerate}
\end{lemma}
\begin{proof}
The first and second assertions are equivalent due to \cite{C--K--M2}, Proposition 2.11. The equivalence of the third and first assertions is \cite{C--K--M2}, Propositions 2.10.

We prove that assertion (c) implies assertion (d). In this case $h^1\big(F,\cE(-th)\big)=0$ for $t=1,2$, hence Formula \eqref{RRGeneral} and Equalities \eqref{eqUlrich} imply
\begin{gather*}
h^0\big(F,\cE(-h)\big)\le\chi(\cE(-h))=0,\\
h^0\big(F,\cE^\vee(2h)\big)=h^2\big(F,\cE(-2h)\big)\le\chi(\cE(-2h))=0.
\end{gather*}

Finally, we prove that assertion (d) implies assertion (c). If $h^0\big(F,\cE(-h)\big)=h^0\big(F,\cE^\vee(2h)\big)=0$, then $h^2\big(F,\cE(-2h)\big)=h^0\big(F,\cE^\vee(2h)\big)=0$. We have
\begin{equation*}
\label{Bounds}
\begin{gathered}
h^0\big(F,\cE(-2h)\big)\le h^0\big(F,\cE(-h)\big)\\
h^2\big(F,\cE(-h)\big)=h^0\big(F,\cE^\vee(h)\big)\le h^0\big(F,\cE^\vee(2h)\big)=h^2\big(F,\cE(-2h)\big).
\end{gathered}
\end{equation*}
It follows that $h^1\big(F,\cE(-th)\big)=-\chi(\cE(-th))$,  for $t=1,2$. Formula \eqref{RRGeneral} and Equalities \eqref{eqUlrich} yield $h^1\big(F,\cE(-th)\big)=0$. We deduce that $\cE$ is Ulrich.
\end{proof}

Notice that if $r=1$, then $\cO_F(D)$ is Ulrich if and only if  $D^2=4(a-1)$, $Dh=3a$ and $h^0\big(F,\cO_F(D-h)\big)=h^0\big(F,\cO_F(2h-D)\big)=0$.

\begin{proposition}
\label{pBound}
Let $\cE$ be an Ulrich bundle of rank $r$ on a $K3$ surface $F\subseteq\p{N}$ of degree $2a$, where $a\ge2$.

Then $c_1(\cE)^2$ is an even integer satisfying
$$
4(a-1)r^2\le c_1(\cE)^2\le\frac{9}{2}ar^2,
$$
where $c_1(\cE)^2\ne\frac{9}{2}ar^2-2$ if $r$ is even. Moreover, $c_1(\cE)^2=\frac{9}{2}ar^2$ if and only if $c_1(\cE)=3rh/2$. 

If $\cE$ is simple, then $(4a-2)r^2-2\le c_1(\cE)^2$.\end{proposition}
\begin{proof}
Thanks to Theorem \ref{tUnstable}, $\cE$ is $\mu$--semistable, thus the Bogomolov inequality (see \cite{H--L}, Theorem 3.4.1) holds for $\cE$. Taking into account Equalities \eqref{eqUlrich} we obtain $c_1(\cE)^2-4(a-1)r^2\ge0$, i.e. $4(a-1)r^2\le c_1(\cE)^2$.

Let $\det(\cE)\cong\cO_F(C)$. Trivially we have $Ch=3ar$ and $C^2=c_1^2$. The Hodge index theorem applied to $h$ and $C$ yields $C^2h^2\le (Ch)^2$, i.e. $c_1(\cE)^2\le\frac{9}{2}ar^2$. Moreover, equality holds if and only if $c_1(\cE)=3rh/2$, because $\Num(F)\cong\Pic(F)$ on a $K3$ surface. Finally recall that $c_1(\cE)^2$ is necessarily even.

Assume now that $\cE$ is also simple, thus the coarse moduli space $\Spl_F(r;c_1(\cE),c_2(\cE))$ of rank $r$ simple, vector bundles on $F$ with Chern classes $c_1(\cE)$ and $c_2(\cE)$ has a non--empty component containing $\cE$.  As pointed out in Theorem 0.1 of \cite{Mu}, thanks to Equalities \eqref{eqUlrich}, such a component has dimension $c_1(\cE)^2-(4a-2)r^2+2 \ge 0$, i.e. $(4a-2)r^2-2\le c_1(\cE)^2$.

Let $\cE$ be an Ulrich bundle of even rank $r=2s$ with $c_1(\cE)^2=\frac{9}{2}ar^2-2$. On the one hand, if $\cO_F(D)\cong\det(\cE)^{-1}(3sh)$, then $D^2=-2$, thus either $D$, or $-D$ should be effective. On the other hand $Dh=0$, thus neither $D$, nor $-D$ can be effective due to the ampleness of $\cO_F(h)$, a contradiction.
\end{proof}

Notice that one can deduce the inequality $c_1(\cE)^2\ne\frac{9}{2}ar^2-2$, also by applying directly Theorem 1.1, (iii) of \cite{Knu2}.

\begin{remark}
\label{rLowerUpper}
The bounds above are sharp in many cases. 

Indeed let $F$ support an Ulrich line bundle $D$, so that $D^2=4(a-1)$ and $Dh=3a$  (see Equalities \eqref{eqUlrich}). Thus $\cE:=\cO_F(D)^{\oplus s}$ is an Ulrich bundle of rank $r:=s$ with $c_1(\cE)^2=4(a-1)r^2$. 

Similarly, every general $K3$ surface $F\subseteq\p{N}$ supports an indecomposable Ulrich bundle $\cF$ of rank $2$ with $c_1(\cE)=3h$ (see \cite{A--F--O}, Theorem 0.4). It is easy to check that $\cE:=\cF^{\oplus s}$ is an Ulrich bundle of rank $r:=2s$ with $c_1(\cE)^2=\frac{9}{2}ar^2$.

Obviously, such bundles are not simple, unless $s=1$ and $\cF$ is indecomposable.
\end{remark}

On the one hand, we will see in the next section that the above bound is optimal when $\cE$ is an Ulrich bundle of rank $r=2$. On the other hand when $\cE$ is an Ulrich bundle of odd rank $r$, the upper bound $c_1(\cE)^2\le9r^2$ of Proposition \ref{pBound} is never sharp, because $D^2$ is even for each divisor $D$ on $F$. Actually, for $r=1$, we can certainly say that such a bound is very far from optimality. Thus the following question raises naturally.

\begin{question}
When $r$ is odd, is there any bound which is sharper than the one given in Proposition \ref{pBound}?
\end{question}

\section{Ulrich bundles of rank $2$ on $K3$ surfaces}
\label{sUlrich2}

The second question raised in the introduction is to prove whether indecomposable Ulrich bundles satisfying the above bounds actually exist. We will give below a partial answer to this question, by constructing explicitly $K3$ surfaces endowed with suitable indecomposable (or, equivalently, simple by Lemma \ref{lRank2}) Ulrich bundles of rank $2$.

Let $\cE$ be an Ulrich bundle on $F$ of rank $2$. We have
$$
c_1(\cE)h=6a,\qquad c_2(\cE)=\frac{c_1^2(\cE)}2-4(a-1)
$$
(see Equalities \eqref{eqUlrich}), hence $\mu(\cE)=3a$. It follows from Proposition \ref{pBound} that $c_1(\cE)^2$ is an even integer satisfying
$$
16(a-1)\le c_1(\cE)^2\le18a,\qquad c_1(\cE)^2\ne18a-2.
$$
We already know that both the cases $c_1(\cE)^2=16(a-1)$ and $18a$ occur (see Remark \ref{rLowerUpper}). 

Assume the existence of an Ulrich bundle $\cE$ of rank $2$ with $c_1(\cE)^2=16a-14$ (resp. $16a-12$). Due to Proposition \ref{pBound}, the bundle $\cE$ is not simple, hence decomposable (see Lemma \ref{lRank2}). If $\cE\cong\cO_F(A)\oplus\cO_F(B)$, then $\cO_F(A)$ and $\cO_F(B)$ are both Ulrich. Equalities \eqref{eqUlrich} yield $A^2=B^2=4(a-1)$: the equality $c_1(\cE)^2=16a-14$ (resp. $16a-12$) finally forces $AB=4a-3$ (resp. $AB=4a-2$). 

If $AB=4a-3$, then $(A-B)^2=-2$, hence either $A-B$, or $B-A$ must be effective. Thus such a case cannot occur because $(A-B)h=0$, due to the ampleness of $\cO_F(h)$.

We will now show that also all the other remaining cases occur, in the sense that there is a $K3$ surface $F\subseteq\p N$ of degree $2a$ with Picard number $\rk(\Pic(F))=3$ supporting Ulrich bundles $\cE$ of rank $2$ with $c_1(\cE)$ such that $16a-12\le c_1(\cE)^2\le 18a$. The above discussion shows that if $c_1(\cE)^2=16a-12$, then $\cE\cong\cO_F(A)\oplus\cO_F(B)$ with $AB=4a-2$. 

The next proposition is the first step in this direction.

\begin{proposition}
\label{pExistence}
Let $a\ge2$. For each choice of an integer $u$ in the range $4a-2\le u\le 5a+2$, there exists a  $K3$ surface $F\subseteq\p N$ of degree $2a$ such that $\Pic(F)$ is freely generated by $h$, $A$, $B$, with $AB=u$, where $\cO_F(A)$  and $\cO_F(B)$ are Ulrich line bundles. 
\end{proposition}
\begin{proof}
We fix the lattice $\Lambda:=\bZ h\oplus\bZ A\oplus \bZ B$ having 
\begin{equation*}
\label{lattice}
M:=\left(
\begin{array}{ccc}
2a  & 3a  & 3a  \\
3a  & 4(a-1)  &  u \\
3a &  u &   4(a-1)
\end{array}
\right)
\end{equation*}
as its intersection matrix. Such a lattice is even. Moreover, it is an easy exercise to check that it has signature $(1,2)$ in the range $4a-3\le u\le 5a+3$. Theorem 2.9 (i) in \cite{Mor} implies the existence of a projective $K3$ surface $F$ having $\Pic(F)\cong \Lambda$ (see also \cite{Ni}). 

Recall that for each divisor $\Gamma$ with $\Gamma^2=-2$ on $F$ we have the Picard--Lefschetz reflection $\pi_{\Gamma}$ of $\Pic(F)$ defined by $D\mapsto D+(D\Gamma)\Gamma$. If $D'$ is another divisor on $F$ then $\pi_\Gamma(D)\pi_\Gamma(D')=DD'$, because $\Gamma^2=-2$.

As pointed out in Proposition VIII.3.9 of \cite{B--H--P--VV}, the cone of big and nef divisors is a fundamental domain for the group generated by the above reflections. In particular we can find divisors $\Gamma_i$ with $\Gamma_i\Gamma_j=-2\delta_{i,j}$, $i=1,\dots,\gamma$, such that 
$$
{h'}:=h+\sum_{i=1}^\gamma(h\Gamma_i)\Gamma_i
$$
is nef. Let
$$
{A'}:=A+\sum_{i=1}^\gamma(A\Gamma_i)\Gamma_i,\qquad {B'}:=B+\sum_{i=1}^\gamma(B\Gamma_i)\Gamma_i.
$$
Then ${h'}$, ${A'}$, ${B'}$ generate $\Pic(F)$ and they still have $M$ as intersection matrix. Omitting the prime in the superscript we can thus assume that $h$ is nef.

We will now show that $h$ is actually very ample. Since $h^2=2a$ it will follow that the surface $F$ can be embedded as a surface of degree $2a$.

Since $h^2=2a\ge4$, thanks to \cite{SD} (see also Theorem 1.1 of \cite{Knu1} with $k=1$, or \cite{Knu2}, Lemma 2.4), we have to check that there are no effective divisors $E$ on $F$ satisfying anyone of the following conditions:
\begin{itemize}
\item
 $E^2=0$ and $Eh=1,2$;
\item  $E^2=2$ and $\cO_F(h)\cong\cO_F(2E)$;
 \item $E^2=-2$ and $Eh=0$. 
\end{itemize}
Notice that, in any case, if $E\in\vert xA+yB+zh\vert$, then $Eh=3ax+3ay+2az$ is a multiple of $a\ge2$. 

Thus the first case can occur only if $a=2$ and, in this case, $Eh=2$ necessarily, i.e. $2z=1-3x-3y$.  Simple computations show that
$$
E^2=-5x^2-(18-2u)xy-5y^2+1.
$$
Consider the ellipse $5x^2+(18-2u)xy+5y^2=1$, where $4a-2=6\le u\le12=5a+2$. The $x$--coordinate intersection point of the ellipse with the line $y=1$ is a root of the polynomial $5x^2+(18-2u)x+4$. 

The discriminant of this polynomial is $\Delta(u):=u^2-18u+61$ which is an increasing function for $u\ge9$ and symmetric around $u=9$. Thus
$$
u^2-18u+61\le \Delta(12)=-11
$$
in the whole range $6\le u\le 12$. 

We conclude that the line $y=1$ has no points in common with the ellipse. Since the ellipse is symmetric with respect to the origin, it immediately follows that it is strictly contained in the square with vertices $(\pm1,0)$ and $(0,\pm1)$: in particular there are no points with integral coordinates on the ellipse, because it does not pass through the origin.

The second case $E^2=2$ cannot occur because $h$ is an element of a basis of $\Pic(F)$ (see the comments after Lemma 2.4 of \cite{Knu2}). 

Thus we look at the third case $E^2=-2$. In this case equality $Eh=0$ implies $2z=-3x-3y$. Simple computations show that
$$
E^2=-\frac{a+8}2x^2-(9a-2u)xy-\frac{a+8}2y^2.
$$
Consider the ellipse $(a+8)x^2+2(9a-2u)xy+(a+8)y^2=4$, where $a\ge2$ and $4a-2\le u\le 5a+2$. Intersecting with the line $y=1$ we obtain $(a+8)x^2+2(9a-2u)x+a+4=0$. The discriminant $\Delta_a(u):=4u^2-36au+80a^2-12a-32$ is an increasing function for $u\ge9a/2$ and it is symmetric around $u=9a/2$. Thus
$$
4u^2-36au+80a^2-12a-32\le \Delta_a(5a+2)=-4(a+4)<0
$$
in the range $4a-2\le u\le 5a+2$. 

We conclude that $\cO_F(h)$ is very ample, hence it embeds $F$ as a surface of degree $2a$. We will now prove that $\cO_F(A)$ and $\cO_F(B)$ are Ulrich line bundles with respect to such an embedding. We restrict our attention to $\cO_F(A)$ because the argument for $\cO_F(B)$ is similar.

Since $A^2=4(a-1)\ge4$, it follows from Equality \eqref{RRGeneral} that either $\cO_F(A)$ or $\cO_F(-A)$ is effective. Since  $Ah=3a\ge6$ and $h$ is very ample it follows that $\cO_F(A)$ is effective. 

Assume that  $h^0\big(F,\cO_F(A-h)\big)\ne0$ and let $D\in\vert A-h\vert$. The divisor $D$ is a curve of degree $\deg(D)=(A-h)h=a$ such that $D^2=(A-h)^2=-4$, thus there is a proper integral subscheme $E\subseteq D$ with $E^2=-2$. Again let $E\in\vert xA+yB+zh\vert$. The degree of $E$ is $\deg(E)=(xA+yB+zh)h=3ax+3ay+2az$, hence it is a positive multiple of $a$. It follows from the chain of inequalities $a\le \deg(E)\le\deg(D)=a$ that $\deg(E)=\deg(D)$, whence $E=D$, a contradiction. We conclude that $h^0\big(F,\cO_F(A-h)\big)=0$.

Notice that $3h-A$ enjoys the same intersections properties with $h$ as $A$. Thus we can repeat the above discussion showing that $h^0\big(F,\cO_F(2h-A)\big)=0$. We conclude that $\cO_F(A)$ is Ulrich, by applying
Lemma \ref{lUlrich}. 
\end{proof}

\begin{remark}
It is not possible to extend the above proofs to the cases $u=4a-3$ and $u=5a+3$. Indeed, in these cases, $(A-B)h=(3h-A-B)h=0$. Moreover $(A-B)^2=-2$ in the first case, and $(3h-A-B)^2=-2$ in the second, thus $\cO_F(h)$ is not very ample. In these cases $\cO_F(h)$ maps birationally $F$ onto a singular surface. 
\end{remark}

We conclude the section with the following consequence of the above proposition. It shows that the bounds of Proposition \ref{pBound} are actually sharp.

\begin{theorem}
\label{tExistence}
Let $a\ge2$. For each choice of an integer $u$ in the range $4a-1\le u\le 5a+4$, $u\ne5a+3$, there exists a $K3$ surface $F\subseteq\p N$ of degree $2a$ and $\rk(\Pic(F))=3$ supporting an indecomposable Ulrich bundle $\cE$ of rank $2$ with $c_1(\cE)^2=8a-8+2u$ and $c_2(\cE)=u$. 
\end{theorem}
\begin{proof}
Let $u$ be an integer in the range $4a-1\le u\le 5a+2$ and $F\subseteq\p N$ a  $K3$ surface of degree $2a$ containing divisors $A$ and $B$ with $AB=u$ as in Proposition \ref{pExistence}. Since $u\ne 4(a-1)=A^2=B^2$, it follows that $\cO_F(A)\not\cong\cO_F(B)$, thus the equalities $(A-B)h=(B-A)h=0$ and the ampleness of $\cO_F(h)$ imply
$$
h^0\big(F,\cO_F(A-B)\big)=0,\qquad h^2\big(F,\cO_F(A-B)\big)=h^0\big(F,\cO_F(B-A)\big)=0.
$$
Equality \eqref{RRGeneral} for $\cO_F(A-B)$ implies
$$
h^1\big(F,\cO_F(A-B)\big)=-2-\frac{(A-B)^2}2=AB-4a+2\ge1
$$

We deduce the existence of the non--split sequence
$$
0\longrightarrow\cO_X(A)\longrightarrow\cE\longrightarrow\cO_X(B)\longrightarrow0,
$$
thus $\cE$ is a rank $2$ Ulrich bundle with $c_1(\cE)^2=(A+B)^2=8a-8+2AB$ and $c_2(\cE)=AB$. Since the above sequence corresponds  to a non--zero element of $H^1\big(F,\cO_F(A-B)\big)$, it follows from Lemma 5.3 of \cite{PL--T} that $\cE$ is simple, hence indecomposable. 
\end{proof}

\begin{remark}
\label{rRk=1,2}
The existence of special Ulrich bundles $\cE$ of rank $2$ (i.e. such that $c_1(\cE)=3h$) on each $K3$ surface $F\subseteq\p{N}$ of degree $2a$ follows from \cite{A--F--O}, Theorem 0.4, when $\rk(\Pic(F))=1$. For each bundle of this type $c_1(\cE)^2=8a-8+2u$ and $c_2(\cE)=u$ where $u=5a+4$.

Similarly, Theorem 4.6 of \cite{Knu2} implies the existence of  $K3$ surfaces $F\subseteq\p N$ of degree $2a$ whose Picard group is freely generated by $h$ and by a smooth irreducible curve $A$ with $Ah=3a$ and $A^2=4(a-1)$. The line bundle $\cO_F(A)$ is Ulrich (one can imitate the proof of the analogous statement in Proposition \ref{pExistence}). As in the proof of Theorem \ref{tExistence} we still obtain special Ulrich bundles by extension of $\cO_F(A)$ by $\cO_F(3h-A)$.
\end{remark}

\begin{remark}
\label{rTriple}
Notice that the $K3$ surface $F\subseteq\p{a+1}$ of degree $2a$ defined in Proposition \ref{pExistence} supports, besides the Ulrich line bundles  $\cO_F(A)$ and $\cO_F(B)$ with $AB=u$, at least another Ulrich line bundle, namely $\cO_F(3h-A)$: we have $B(3h-A)=9a-u$. 

In particular for $4a-1\le u \le 9a/2$ there is a  $K3$ surface $F\subseteq\p N$ of degree $2a$ supporting indecomposable Ulrich bundles $\cE$ of rank $2$ with $c_1(\cE)^2=18a, 8a-8+2u, 26a-8-2u$.
\end{remark}

\begin{remark}
\label{rSemistable}
Each Ulrich bundle is semistable (see Theorem \ref{tUnstable}). Trivially the bundle $\cE$ constructed in Theorem \ref{tExistence} is strictly semistable, i.e. is semistable and not stable, because it contains an Ulrich line bundle $\cO_F(A)$ with $\mu(\cO_F(A))=3a=\mu(\cE)$.
\end{remark}

We showed that for each integer $u$ in the range $4a-1\le u\le 5a+4$, $u\ne 5a+3$ there exists a $K3$ surface $F\subseteq\p{a+1}$ of degree $2a$ with $\rk(\Pic(F))=3$ supporting Ulrich bundles $\cE$ of rank $2$ with $c_1(\cE)^2=8a-8+2u$ (see Theorem  \ref{tExistence}). It is obvious that such bundles cannot exist if $\rk(\Pic(F))=1$. We thus ask the following question.

\begin{question}
If $u$ is a fixed integer in the range $4a-1\le u\le 5a+4$, $u\ne 5a+3$, is there a $K3$ surface $F\subseteq\p{N}$ of degree $2a$ with $\rk(\Pic(F))=2$ supporting an Ulrich bundle $\cE$ of rank $2$ with $c_1(\cE)^2=8a-8+2u$?
\end{question}

As we already pointed out in Remark \ref{rRk=1,2}, the above question has an immediate positive answer when $u=5a+4$. 

Moreover, for each integer $u$ in the range $4a-1\le u\le 9a/2$ there exist $K3$ surfaces $F\subseteq\p{a+1}$ of degree $2a$ supporting Ulrich bundles $\cE$ of rank $2$ with $c_1(\cE)^2=18a, 8a-8+2u, 26a-8-2u$ (see Theorem \ref{tExistence} and Remark \ref{rTriple}). It is quite natural to raise the following questions.

\begin{question}
Does there exist a single $K3$ surface $F\subseteq\p{N}$ of degree $2a$ supporting Ulrich bundles $\cE$ of rank $2$ with $c_1(\cE)^2=8a-8+2u$ for each $u$ in the range $4a-1\le u\le 5a+4$, $u\ne 5a+3$?
\end{question}

\begin{question}
If the answer to the previous question is positive, which is the minimal admissible value of $\rk(\Pic(F))$?
\end{question}

\section{Stability of general Ulrich bundles of rank $2$}
\label{sStable}
If $\cE$ is a semistable bundle of rank $2$ with reduced Hilbert polynomial $p(t)$ (with respect to $\cO_F(h)$), then the coarse moduli space $\cM_F^{ss}(p)$ parameterizing $S$--equivalence classes of semistable rank $2$ bundles on $F$ with reduced Hilbert polynomial $p(t)$ is non--empty (see Section 1.5 of \cite{H--L} for the notion of $S$--equivalence). We will denote by $\cM_F^{s}(p)$ the open locus inside $\cM_F^{ss}(p)$ of stable bundles. 

The scheme $\cM_F^{ss}(p)$ is the disjoint union of open and closed subsets $\cM_F^{ss}(2;c_1,c_2)$ whose points represent $S$--equivalence classes of semistable rank $2$ bundles with fixed Chern classes $c_1$ and $c_2$. Similarly $\cM_F^{s}(p)$ is the disjoint union of open and closed subsets $\cM_F^{s}(2;c_1,c_2)$. 

The Grauert semicontinuity theorem for complex spaces (see \cite{B--H--P--VV}) guarantees that the property of being aCM in a family of vector bundles is an open condition. 

Thus, on the one hand, we have open subschemes $\cM_F^{ss,aCM}(2;c_1,c_2)\subseteq \cM_F^{ss}(2;c_1,c_2)$ and $\cM_F^{s,aCM}(2;c_1,c_2)\subseteq \cM_F^{s}(2;c_1,c_2)$ parameterizing respectively $S$--equivalence classes of semistable and stable aCM bundles of rank $2$ on $F$ with Chern classes $c_1$ and $c_2$ (see Section 2 of \cite{C--H2}). 

On the other hand, the locus of aCM bundles $\Spl_F^{aCM}(r;c_1,c_2)$ inside the moduli space $\Spl_F(r;c_1,c_2)$ of simple vector bundles of rank $r$ on $F$ with Chern classes $c_1$ and $c_2$ is open too. If we denote by $\Spl_F^{ns,aCM}(2;c_1,c_2)\subseteq\Spl_F^{aCM}(2;c_1,c_2)$ the locus of simple aCM bundles which are not stable, then $\Spl_F^{aCM}(2;c_1,c_2)\setminus \Spl_F^{ns,aCM}(2;c_1,c_2)$ is an open subset isomorphic to $\cM_F^{s,{aCM}}(2;c_1,c_2)$.

We are interested in dealing with the moduli spaces of aCM bundles of rank $2$ on $F$: in particular we are interested in those ones constructed in the previous section.

\begin{proposition}
\label{pStable}
Let $F\subseteq\p{N}$ be a $K3$ surface of degree $2a$, where $a\ge2$.

If $\cO_F(A)$ and $\cO_F(B)$ are Ulrich line bundles on $F$ such that $4a-1\le AB\le 5a+4$, $AB\ne5a+3$, then the following assertions hold.
\begin{enumerate}[label=(\alph*)]
\item The moduli space $\cM^{ss,{aCM}}_F(2;A+B,AB)$ is non--empty. 
\item If $AB\ge4a$, for each irreducible component $\cM\subseteq \cM^{ss,{aCM}}_F(2;A+B,AB)$, the locus $\cM\cap \cM^{s,{aCM}}_F(2;A+B,AB)$ is smooth and non--empty of dimension $2AB-8a+2$.
\item The points in $\cM^{ss,{aCM}}_F(2;A+B,AB)\setminus\cM^{s,{aCM}}_F(2; A+B,AB)$ are in one--to--one correspondence with the unordered  pairs $\{\ \cO_F(\overline{A}),\cO_F(\overline{B})\ \}$ where $\cO_F(\overline{A})$ and $\cO_F(\overline{B})$ are Ulrich line bundles such that $\cO_F(\overline{A}+\overline{B})\cong\cO_F(A+B)$.
\end{enumerate}
\end{proposition}
\begin{proof}
Notice that Lemma \ref{lUlrich} implies that the points in $\Spl_F^{aCM}(2;A+B,AB)$ actually parameterize rank $2$ simple Ulrich bundles with Chern classes $A+B$ and $AB$.

We have that $\cO_F(A)\ne\cO_F(B)$, because $AB> 4(a-1)=A^2$. The locus $\Spl_F^{ns,aCM}(2;A+B,AB)$ is non--empty, due to Theorem \ref{tExistence} and Remark \ref{rSemistable}, hence $\Spl_F^{aCM}(2;A+B,AB)$ is smooth, non--empty and its dimension is $4c_2-c_1^2-6=2AB-8a+2$ (see \cite{Mu}, Theorem 0.1). 

If $\cE$ is a bundle representing a strictly semistable Ulrich bundle in $\Spl_F^{aCM}(2;A+B,AB)$, then $\cE$ must contain a line bundle $\mathcal L$ such that $\cE/\mathcal L$ is torsion free and $\mu(\mathcal L)=\mu(\cE)$. Theorem \ref{tUnstable} implies that $\cE$ must fit into a sequence of the form
\begin{equation}
\label{seqOverline}
0\longrightarrow\cO_F(\overline{A})\longrightarrow\cE\longrightarrow\cO_F(\overline{B})\longrightarrow0.
\end{equation}
where $\overline{A}$ and $\overline{B}$ are Ulrich line bundles such that $\overline{A}+\overline{B}=A+B$ and $\overline{A}\, \overline{B}=AB$. Thus we obtain a non--zero element of $H^1\big(F,\cO_F(\overline{A}-\overline{B})\big)\cong H^1\big(F,\cO_F(2\overline{A}-A-B)\big)$.

As a first consequence of the above discussion we are able to prove assertions (a) and (b). We start with assertion (a). We see that $\Spl_F^{ns,aCM}(2;A+B,AB)$  is dominated by the union of the projective spaces associated to $H^1\big(F,\cO_F(2\overline{A}-A-B)\big)$ as $\overline{A}$ varies in the subset of Ulrich line bundles inside $\Pic(F)$. Such a set is trivially countable, thus $\Spl_F^{ns,aCM}(2;A+B,AB)$ is a countable union of irreducible subschemes of dimension at most $h^1\big(F,\cO_F(\overline{A}-\overline{B})\big)=AB-4a+1$: in particular, when $AB\ge4a$, it cannot fill any irreducible component $\mathcal S\subseteq \Spl_F^{aCM}(2;A+B,AB)$ (which is smooth of dimension  $4c_2-c_1^2-6=2AB-8a+2$), thus there exist semistable, aCM bundles $\cE$ of rank $2$ with Chern classes $c_1(\cE)=A+B$ and $c_2(\cE)=AB$. In particular assertion (a) is proved.

Now we prove assertion (b). Since $\Spl_F^{aCM}(2;A+B,AB)$ contains $\cM^{s,aCM}_F(2;A+B,AB)$ as an open subset, it follows that $\mathcal S\setminus \Spl_F^{ns,aCM}(2;A+B,AB)$ is an irreducible component of $\cM^{s,aCM}_F(2;A+B,AB)$. We deduce that the $S$--equivalent class of each $\cE\in\mathcal S\cap \Spl_F^{ns,aCM}(2;A+B,AB)$ is actually in the closure of a non--empty component of $\cM^{s,aCM}_F(2;A+B,AB)$. Thus assertion (b) is proved.

As a second consequence of the above discussion we prove assertion (c). Indeed, if $\cE$ fits into Sequence \eqref{seqOverline}, its Jordan--H\"older filtration 
is $0\subseteq\cO_F(\overline{A})\subseteq\cE$, thus the associated graded ring is
$$
\gr(\cE):=\cO_F(\overline{A})\oplus\cE/\cO_F(\overline{A})\cong\cO_F(\overline{A})\oplus\cO_F(\overline{B}).
$$

Let $\cE'$ be in the same $S$--equivalence class of $\cE$. In particular $\cE'$ is strictly semistable, hence it fits into a sequence of the form
$$
0\longrightarrow\cO_F(\overline{A}')\longrightarrow\cE'\longrightarrow\cO_F(\overline{B}')\longrightarrow0,
$$
where $\cO_F(\overline{A}')$, $\cO_F(\overline{B}')$ are Ulrich line bundles which are non--isomorphic. Thus
$$
\cO_F(\overline{A}')\oplus\cO_F(\overline{B}')\cong\gr(\cE')\cong\gr(\cE)\cong\cO_F(\overline{A})\oplus\cO_F(\overline{B}).
$$

We have a non--zero morphism $\cO_F(\overline{A})\to\cO_F(\overline{A}')\oplus\cO_F(\overline{B}')$, thus either $h^0\big(F,\cO_F(\overline{A}'-\overline{A})\big)\ne0$, or $h^0\big(F,\cO_F(\overline{B}'-\overline{A})\big)\ne0$. In the first case, the equality $(\overline{A}'-\overline{A})h=0$ and the ampleness of $\cO_F(h)$ imply $\cO_X(\overline{A}-\overline{A})\cong\cO_X$, hence the above non--zero map induces an isomorphism  $\cO_F(\overline{A})\cong\cO_F(\overline{A}')$, thus 
$$
\cO_F(\overline{B}')\cong\cO_F(A+B-\overline{A}')\cong\cO_F(A+B-\overline{A})\cong\cO_F(\overline{B}),
$$
i.e. 
$\{\ \cO_F(\overline{A}),\cO_F(\overline{B})\ \}=\{\ \cO_F(\overline{A}'),\cO_F(\overline{B}')\ \}$.

An analogous argument holds if $h^0\big(F,\cO_F(\overline{B}'-\overline{A})\big)\ne0$. Thus assertion (c) is proved too.
\end{proof}

The above proposition yields the following theorem (see also Proposition 3.22 of \cite{C--K--M1}).

\begin{theorem}
\label{tStable}
Let $F\subseteq\p{N}$ be a $K3$ surface of degree $2a$, where $a\ge2$.

If $\cE$ is an indecomposable  Ulrich bundle of rank $2$ on $F$ which is  strictly semistable and whose $S$--equivalence class is a general point in its moduli space, then $\cE$ fits into Sequence \eqref{seqAB} where  $\cO_F(A)$ and $\cO_F(B)$ are Ulrich line bundles on $F$ such that $AB=4a-1$. In particular $c_1(\cE)^2=16a-10$ and $c_2(\cE)=4a-1$.
\end{theorem}
\begin{proof}
If $\cE$ is strictly semistable, then it contains a line bundle $\cO_F(A)$ with $\mu(\cO_F(A))=\mu(\cE)=3a$. Trivially $\mu(\cE/\cO_F(A))=3a$: due to Theorem \ref{tUnstable} it follows that $\cO_F(A)$ and $\cE/\cO_F(A)$ are Ulrich line bundles. Then $\cE$ fits into Sequence \eqref{seqAB}.

If $\cE$ is also general in its moduli space, then Proposition \ref{pStable} forces $AB=4a-1$.
\end{proof}

The construction of Section \ref{sUlrich2} yields the existence of semistable Ulrich bundles of rank $2$ with $c_1(\cE)^2=8a-8+2u$ for each integer $u$ in the range $4a-1\le u\le 5a+4$, $u\ne 5a+3$. 

On the one hand, Theorem \ref{tStable} implies that when $u\ge4a$, the general such bundle is actually stable. On the other hand, when $u=4a-1$, Proposition \ref{pStable} implies that the moduli space is a finite set of points and we know that at least one of these points corresponds to a strictly semistable bundle. 

\begin{question}
Are there $K3$ surfaces $F\subseteq\p{N}$ of degree $2a$ supporting a stable Ulrich bundle $\cE$ of rank $2$ with $c_1(\cE)^2=16a-10$?
\end{question}

\bigskip
\noindent
Gianfranco Casnati,\\
Dipartimento di Scienze Matematiche, Politecnico di Torino,\\
c.so Duca degli Abruzzi 24, 10129 Torino, Italy\\
e-mail: {\tt gianfranco.casnati@polito.it}
\bigskip

\noindent
Federica Galluzzi,\\
Dipartimento Matematica, Universit\ga a degli Studi di Torino,\\
via Carlo Alberto 10, 10123 Torino, Italy\\
e-mail: {\tt federica.galluzzi@unito.it}


\begin{thebibliography}{44}

\bibitem{A--F--O}
 M. Aprodu, G. Farkas, A. Ortega: \emph{Minimal resolutions, Chow forms and Ulrich bundles on $K3$ surfaces}. Preprint arXiv:1212.6248 [math.AG]. To appear in  J. Reine Angew. Math.
 
\bibitem{B--H--P--VV}
W. Barth, K. Hulek, Ch. Peters, A. van de Ven: \emph{Compact complex surfaces}. Second edition. Springer \rm(2004). 

\bibitem{C--H1}
M. Casanellas, R. Hartshorne: \emph{ACM bundles on cubic surfaces}. J. Eur. Math. Soc. \textbf{13} \rm (2011), 709--731.

\bibitem{C--H2}
  M. Casanellas, R. Hartshorne, F. Geiss, F.O. Schreyer: \emph{Stable Ulrich bundles}. Int. J. of Math. \textbf{23} 1250083 \rm(2012).

\bibitem{Cs1}
G. Casnati: \emph{On rank two aCM bundles}. To appear in Comm. Algebra.

\bibitem{Cs2}
G. Casnati: \emph{Rank two aCM bundles on general determinantal quartic surfaces in $\p3$}. Preprint arXiv:1601.02911 [math.AG]. 

\bibitem{Cs3}
G. Casnati: \emph{Examples of smooth surfaces in $\p3$ which are Ulrich--wild}. Preprint.

\bibitem{C--N}
G. Casnati, R. Notari: \emph{Examples of rank two aCM bundles on smooth quartic surfaces in $\p3$}. Preprint arXiv:1601.02907 [math.AG]. To appear in  Rend. Circ. Mat. Palermo.

\bibitem{Co}
  E. Coskun: \emph{Ulrich bundles on quartic surfaces with Picard number $1$}. C. R. Acad. Sci. Paris Ser. I \textbf{351} \rm (2013), 221--224.

\bibitem{C--K--M1}
  E. Coskun, R.S. Kulkarni, Y. Mustopa: \emph{Pfaffian quartic surfaces and representations of Clifford algebras}. Doc. Math.  \textbf{17} \rm (2012), 1003--1028.

\bibitem{C--K--M2}
  E. Coskun, R.S. Kulkarni, Y. Mustopa: \emph{The geometry of Ulrich bundles on del Pezzo surfaces}. J. Algebra \textbf{375} (2013), 280--301.

\bibitem{E--He}
D. Eisenbud, J. Herzog: \emph{The classification of homogeneous Cohen--Macaulay rings of finite representation type}. Math. Ann. \textbf{280} \rm(1988), 347--352.

\bibitem{E--S--W}
D. Eisenbud, F.O. Schreyer, J. Weyman: \emph{Resultants and Chow forms via exterior syzygies}.  J. Amer. Math. Soc. \textbf{16} \rm(2003), 537--579.

\bibitem{Fa}
D. Faenzi: \emph{Rank $2$ arithmetically Cohen--Macaulay bundles on a nonsingular cubic surface}. J. Algebra \textbf{319} (2008), 143--186.

\bibitem{H--L}
D. Huybrechts, M. Lehn: \emph {The geometry of moduli spaces of sheaves. Second edition}. Cambridge Mathematical Library, Cambridge U.P. \rm (2010).

\bibitem{Knu1}
A. Knutsen: \emph{On $k$th--order embeddings of $K3$ surfaces and
Enriques surfaces}. Manuscripta math. \textbf{104} (2001), 211--237.

\bibitem{Knu2}
A. Knutsen: \emph{Smooth curves on projective $K3$ surfaces}. Math. Scand. \textbf{90} (2002), 215--231.

\bibitem{Mor}
D. Morrison:  \emph{On K3 surfaces with large Picard number}. Invent. Math. \textbf{75} (1984), 105--121.

\bibitem{Mu}
S. Mukai: \emph{Symplectic structure of the moduli space of sheaves on an abelian or $K3$ surface}.  Invent. Math. \textbf{77} (1984), 101--116.
 
\bibitem{Ni} 
V. Nikulin:  \emph{Integral symmetric bilinear forms and some of their applications}. Math.USSR--Izv. \textbf{14} (1980), 103--167.
 
\bibitem{O--S--S}
  C. Okonek, M. Schneider, H. Spindler: \emph{Vector bundles on complex projective spaces}. Progress in Mathematics 3, \rm(1980).
  
  \bibitem{PL--T}
  J. Pons-Llopis, F. Tonini: \emph{ACM bundles on del Pezzo surfaces}. Matematiche (Catania) \textbf{64} (2009), 177--211.

\bibitem{SD}
B. Saint--Donat: \emph{Projective models of $K$--$3$ surfaces}. Amer. J. Math. \textbf{96} \rm(1974), 602--639.

\bibitem{Wa}
K. Watanabe: \emph{The classification of ACM line bundles on quartic hypersurfaces in $\p3$}. Geom. Dedicata
\textbf{175} \rm (2015) 347--353.

\bibitem{Wa2}
K. Watanabe: \emph{ACM bundles on $K3$ surfaces of genus $2$}. Preprint arXiv:1407.1703 [math.AG].

\end{thebibliography}
\end{document}